\documentclass[11pt]{amsart}
\usepackage{amsfonts,amsmath,amssymb}

\newtheorem{theorem}{Theorem}[section]
\newtheorem{lemma}[theorem]{Lemma}

\newtheorem{definition}{Definition}[section]
\newtheorem{corollary}[theorem]{Corollary}
\newtheorem{remark}[theorem]{Remark}

\newcommand{\cl}[1]{\mathcal{#1}} 

\newcommand{\sca}[1]{\left\langle#1\right\rangle} 
\newcommand{\Map}[1]{\mathrm{Map}(#1)} 
\newcommand{\Ref}[1]{\mathrm{Ref}(#1)} 
\newcommand{\Lat}[1]{\mathrm{Lat}(#1)} 
\newcommand{\Alg}[1]{\mathrm{Alg}(#1)} 

\title{Stable properties of hyperrelexivity}
\author{G.K. ELEFTHERAKIS}

\address{}
\email{gelefth@math.upatras.gr}

\begin{document}

\maketitle

\begin{abstract}Recently a new equivalence relation between weak* closed operator spaces acting on Hilbert spaces has appeared. Two weak* 
closed operator spaces $\cl U, \cl V$ are called weak TRO equivalent if there exist ternary rings of operators $\cl M_i, i=1,2$ such that 
$ \cl U=[\cl M_2\cl V\cl M_1^*]^{-w^*} , \cl V=[\cl M_2^*\cl U\cl M_1]^{-w^*} .$
Weak TRO equivalent 
spaces  are stably isomorphic, and conversely, stably isomorphic dual operator spaces have normal completely isometric representations 
with weak TRO equivalent images. In this paper, we prove that if $\cl U$ and $\cl V$ are weak TRO equivalent operator spaces 
 and the space of $I\times I$ matrices with entries in $\cl U, M_I^w(\cl U),$ is hyperreflexive for suitable infinite I, 
then so is 
 $M_I^w(\cl V).$ 
We describe situations where if $\cl L_1, \cl L_2$ are isomorphic lattices, then the corresponding algebras 
$\Alg{\cl L_1}, \Alg{\cl L_2} $ have the same complete hyperreflexivity constant.
\end{abstract}

\section{Introduction}

Recently a new equivalence relation between weak* closed operator spaces acting on Hilbert spaces has appeared:

\begin{definition}\label{11}\cite{ept} Let $H_i, K_i, i=1,2$ be Hilbert spaces, and $\cl U\subset B(K_1, K_2), \cl V\subset B(H_1, H_2)$ be  
weak* closed spaces. 
We call them weak 
TRO equivalent if there exist ternary rings of operators (TRO's) $\cl M_i\subset B(H_i, K_i), i=1,2$, i.e. spaces satisfying 
$\cl M_i\cl M_i^*\cl M_i\subset M_i, i=1,2$, such that 
$$ \cl U=[\cl M_2\cl V\cl M_1^*]^{-w^*} , \cl V=[\cl M_2^*\cl U\cl M_1]^{-w^*} .$$
\end{definition}

This notion is related to the very important notion of stable isomorphism of operator spaces:

\begin{theorem}\label{12}\cite{ept} If $\cl U$ and $\cl V$ are weak TRO equivalent operator spaces then $\cl U$ and $\cl V$ 
are weakly stably isomorphic. 
This means that there exists a cardinal $I$ such that the spaces $ \cl U \bar \otimes  B(l^2(I)) , \cl V\bar \otimes  B(l^2(I)) $ 
are completely isometrically isomorphic through a weak* continuous map. Here, $\bar \otimes  $ is the 
normal spatial tensor product. Conversely, if $\cl U$ and $\cl V$ are weakly stably isomorphic, then they 
have completely isometric weak* continuous 
representations $\phi $ and $\psi $ such that $\phi (\cl U)$  and $\psi (\cl V)$ are weak TRO equivalent.
\end{theorem}

In Section 3  of this paper, we prove that if $\cl U$ and $\cl V$ are weak TRO equivalent operator spaces,  
 and if the weak* closed space of $I\times I$ matrices with entries in $\cl U, M_I^w(\cl U)$  is hyperreflexive 
for suitable infinite I, then so is 
 $M_I^w(\cl V).$  In the case of separably acting $\cl U$ and $\cl V$, we have $k(M_\infty ^w(\cl U))=k(M_\infty ^w(\cl U))$  
 where $k(\cl X)$ is the hyperreflexivity constant of $\cl X$ and $\infty$ is aleph 0. As a consequence, in Section 4 we prove that if $\cl A$ and $\cl B$ are stably isomorphic CSL algebras acting 
on separable Hilbert spaces and if $\cl A$ is completely hyperreflexive, then $\cl B$ is also a completely hyperreflexive space with the 
same complete hyperreflexivity constant. We also prove that if 
 $\cl L_i\subset B(H_i), i=1,2$ are separably acting reflexive lattices and there exists a *-isomorphism 
$\theta : \cl L_1^{\prime \prime} \rightarrow \cl L_2^{\prime \prime} $ such that $\theta (\cl L_1)=\cl L_2$, then 
$\Alg{\cl L_1} $ is completely hyperreflexive iff $\Alg{\cl L_2} $ is completely hyperreflexive. 
We also prove that if $\cl L_i\subset B(H_i), i=1,2$ are  totally atomic separably acting isomorphic CSL's, then 
$\Alg{\cl L_1} $ is completely hyperreflexive iff $\Alg{\cl L_2} $ is completely hyperreflexive. 
Finally, we prove that separably acting von Neumann algebras with isomorphic commutants have the same complete  
hyperreflexive constant.

\bigskip

In what follows, the symbol $[\cl{S}]$ denotes the linear span of $\cl{S}.$ If $\cl{L}\subset B(H)$, we 
denote by $\cl{L}^\prime$ the set of operators which commute with the elements of 
$\cl{L}$. The set of projections in $\cl{L}$ is written $pr(\cl{L}).$ If $T$ is an operator and $I$ is a cardinal, 
 $T^I $ denotes the $I\times I$ diagonal matrix with diagonal entries $T.$  If $\cl X$ is a space of operators, we define $\cl X^I  $ to be the space 
containing all operators of the form 
$T^I $ where $T\in \cl X.$

A set of projections of a Hilbert space is called a lattice if it contains the zero and identity projections and 
is closed under arbitrary suprema and infima. If $\cl{A}$ is a subalgebra of $B(H)$ for some Hilbert space $H$, the set
$$\Lat{\cl A }=\{L\in pr(B(H)): L^\bot \cl{A}L=0\}$$
is a lattice. Dually, if $\cl{L}$ is a lattice, the space
$$\Alg{\cl L}=\{A\in B(H): L^\bot AL=0 \;\;\forall\;\; L\in \cl{L}\}$$
is an algebra. A lattice $\cl{L}$ is called reflexive if $\Lat {\Alg{\cl L}}=\cl L.$

 A commutative subspace lattice  (CSL) is  a projection
lattice $\cl{L}$ whose elements commute; the algebra $\mathrm {Alg}(\cl{L})$ is called a CSL algebra. Two CSL's 
$\cl L_1, \cl L_2$ are called isomorphic if there exists an order preserving  1-1 and onto map from $\cl L_1$ onto $\cl L_2.$

\medskip

 Let $H_1, H_2$ be Hilbert spaces and $\cl{U}$ a subset of $B(H_1,H_2).$ The reflexive hull of $\cl{U}$ is defined  to be the space
$$\Ref{ \cl U}=\{T\in B(H_1,H_2): Tx\in \overline{[\cl{U}x]} \;\text{for\;each}\;x\in H_1\}.$$
A subspace $\cl{U}$ is called reflexive if $\cl{U}=\Ref{\cl U}.$  
\medskip

Let $\cl U$ be a subspace of $\cl B(H,K)$. If $T\in \cl B(H,K)$, we call
$$d(T, \cl U) =\inf_{ X\in \cl U }\|T-X\|$$ 
the distance from $T$ to $\cl U$. 
We also set
$$r_{\cl U}(T) =\sup_{\|\xi \|=\|\eta \|=1} \{|\sca{T\xi ,\eta }|: \sca{U\xi, \eta }=0\;\;\forall \;U\in \cl U\}.$$
Trivially, $r_{\cl U}(T)\leq d(T,\cl U)$. 
We can see that $\cl U$ is reflexive if $r_{\cl U}(T) = 0$ implies that $T\in \cl U$, 
for $T\in \cl B(H,K)$. 
If there exists $k>0$ such that 
$$d(T, \cl U) \leq k r_{\cl U}(T), \ \ \ T\in  B(H,K),$$
we say that the space $\cl U$ is hypereflexive.
The space $\cl U$ is called completely hyperreflexive if
$\cl U\bar\otimes B(\cl H)$ is hyperreflexive, where $ H$ is a separable infinite dimensional Hilbert space.
 It is not known if hyperreflexivity implies complete hyperreflexivity.

If $\cl U$ is a reflexive space, let 
$$k(\cl U)=\sup_{T\not \in  \cl U}\frac{d(T, \cl U) }{r_{\cl U}(T) }$$ 
be the the hyperreflexivity constant of $\cl U$.
Clearly, $\cl U$ is hyperreflexive if and only if $k(\cl U)<\infty.$ 

Throughout this paper we shall use the following Lemma.

\begin{lemma}\label{13} Let $\cl U\subset B(K_1, K_2)$ be a weak* closed space, $K_1, K_2$ Hilbert spaces,  
and $\cl B$ and $\cl A$ von Neumann algebras such that 
$\cl B\cl U\cl A\subset \cl U.$ Then for every $T\in  B(K_1, K_2),$  
$$r_\cl U(T)=\sup \{\|QTP\|: Q\in pr(\cl B^\prime), P\in pr(\cl A^\prime), Q\cl UP=0\}.$$
\end{lemma}
\begin{proof} Choose $T\in  B(K_1, K_2)$  and $Q\in pr(\cl B^\prime), P\in pr(\cl A^\prime)$ such that $ Q\cl UP=0.$
We have $$\|QTP\|= \sup_{\|\xi \|=\|\eta \|=1}|\sca{QTP\xi ,\eta }| =\sup_{\|\xi \|=\|\eta \|=1}|\sca{TP\xi ,Q\eta }| .$$
Since $\sca{UP\xi ,Q\eta }=0, \forall U\in \cl U, \xi ,\eta $ we have 
 $$|\sca{TP\xi ,Q\eta }|\leq r_\cl U(T)\|P\xi \|\|Q\eta \|\leq r_\cl U(T).$$

For the converse inequality,  suppose $\epsilon >0$.  Then there exist unit vectors $\xi, \eta $ such that 
$\sca{U\xi ,\eta }=0,   \;\;\forall \;\;U\in \cl U$ and 
$$r_\cl U(T)-\epsilon < |\sca{T\xi ,\eta }| .$$ Since $\sca{UA\xi ,B\eta }=0, \;\;\forall \;A\in \cl A, \;\;\forall \;B\in \cl B$ 
if $P$ is the projection onto the space generated by $\cl A\xi $ and $Q$ is the projection onto the space generated by $\cl B\eta $, 
we have $ Q\cl UP=0$ and $Q\in pr(\cl B^\prime), P\in pr(\cl A^\prime).$ But $$|\sca{T\xi ,\eta }| =|\sca{QTP\xi ,\eta }|\leq \|QTP\| .$$ 
Since $\epsilon $ is arbitrary, the proof is complete.
\end{proof}

\medskip

Now we present some concepts introduced in \cite{erd}.

Let $\cl{P}_i=pr(B(H_i)), i=1,2.$ Let $\phi=\mathrm{Map} (\cl{U})$ be the map $\phi:\cl{P}_1\rightarrow \cl{P}_2$, which to each
$P\in\cl{P}_1$ associates the projection onto the subspace
$[TPy:T\in\;\cl{U},y\in\;H_1]^{-}.$ The map $\phi$ is $\vee-$continuous (that is, it preserves arbitrary suprema) and is $0$ preserving.

Let
$\phi^*=\mathrm{Map} (\mathcal{U}^*),\mathcal{S}_{1,\phi}=\{\phi^*(P)^{\bot}:P\in\cl{P}_2\},\mathcal{S}_{2,\phi}=\{\phi(P):P\in\cl{P}_1\}$ and observe that $\cl{S}_{1,\phi }=\cl{S}_{2,\phi ^*}^\bot .$
Erdos proved that $\mathcal{S}_{1,\phi}$\;is $\wedge $-complete and
contains the identity projection, $\mathcal{S}_{2,\phi}$\;is $\vee $-complete and contains the zero projection, while
$\phi|_{\mathcal{S}_{1,\phi}}:\mathcal{S}_{1,\phi}\rightarrow\mathcal{S}_{2,\phi}$\;is
a bijection. 

In fact, $$\mathrm{Ref} (\cl{U})=\{T\in\;B(H_1,H_2):\phi(P)^{\bot}TP=0\; \text{for
each}\;P\in\mathcal{S}_{1,\phi}\}.$$
When $\phi(I)=I$ and $\phi^*(I)=I$, we call the space $\cl{U}$ essential.

\medskip

In \cite{kt} it is proved that a TRO $\cl{M}$ is weak* closed if and only if  it is $wot$ closed if and only if  it is
reflexive. In this case, if $\chi =\Map{\cl M}$,
$$\cl{M}=\{T\in B(H_1,H_2): TP=\chi (P)T\;\text{for\;all\;}P\in \cl{S}_{1,\chi }\}.$$
 In the following theorem we isolate some consequences of \cite[Theorem 2.10]{kt}.

\begin{theorem}\label{14}(i) A TRO $\cl{M}$ is essential if and only if the algebras 
$[\cl{M}^*\cl{M}]^{-w^*}$\\$ [\cl{M}\cl{M}^*]^{-w^*}$ contain the identity operators.

(ii) If $\cl{M}$ is an essential TRO and $\chi =\Map{\cl M}$, then
$\cl{S}_{1,\chi }=pr((\cl{M}^*\cl{M})^\prime),\\ \cl{S}_{2,\chi }=pr((\cl{M}\cl{M}^*)
^\prime)$ and the map $\chi |_{\cl{S}_{1,\chi }}: \cl{S}_{1,\chi }\rightarrow \cl{S}_
{2,\chi }$ is an ortholattice isomorphism with inverse $\chi^* |_{\cl{S}_{2,\chi }}.$

\end{theorem}

\medskip

If $K_1, K_2$ are Hilbert spaces, $\cl U\subset B(K_1, K_2)$ is a weak* closed operator space, and $I$ is a cardinal, then $M^w_I (\cl U)$ is the set of 
$I \times I $ matrices with entries in $\cl U$ whose finite submatrices have uniformly bounded 
norm, \cite{bm}. We consider $M^w_I (\cl U)$  as a subspace of the set of bounded operators from $K_1^I $ to 
$K_2^I.$ 
We can see that 
the space $M^w_I (\cl U)$  is unitarily equivalent with $ \cl U\bar \otimes B(l^2(I)) .$ 
Therefore if 
$\cl U$ is a completely hyperreflexive space, then $k(  \cl U\bar \otimes B(l^2(\mathbb N))  )=k(M^w_\infty (\cl U) ).$
 Also, $ C_I ^w(\cl U) $ is the subspace of $I\times 1$ columns with entries in $\cl U$, or, equivalently, the space of bounded operators from $K_1$ to $K_2^I $ of the form $(U_i)_{i\in I}$, where 
every $U_i$ belongs to $\cl U.$ 
\begin{lemma} $$k(C_I ^w(\cl U) )\leq k(M_I^w(\cl U)).$$
\end{lemma}
\begin{proof} We denote by $E=(E_{i,j})_{i,j\in I}$ the $I\times I$ matrix where $E_{i_0,i_0}=I_{K_1}$ 
and $E_{i,j}=0$ for $(i,j)\neq (i_0,i_0).$ Observe that $ M_I^w(\cl U)E $ contains elements of the form $(C_i)_{i\in I},$ 
where $C_{i_0}=C_I^w(\cl U)$ and $C_i$ is a zero column for $i\neq i_0.$ 

Lemma 6.2 in \cite{dale} implies that $k(M_I^w(\cl U)E )\leq k(M_I^w(\cl U) ).$ Obviously 
$$k(C_I^w(\cl U))\leq k(M_I^w(\cl U)E ).$$  
\end{proof}

\medskip

In this paper we shall use the following lemma from 8.5.23 in \cite{bm}.

\begin{lemma} \label{111}If $\cl M\subset B(K_1, K_2)$ is an essential weak* closed TRO,  and $K_1, K_2$ are  Hilbert spaces,  
and $I$ is the cardinal of an orthonormal basis of $K_1$, there exists a column $M=(M_i)_{i\in I}\in 
C_I ^w(\cl M)  $ where 
every $M_i$ is a partial isometry such that $M_i^*M_i$ is orthogonal to $M_j^*M_j$ for every $i\neq j$ and such that 
 $M^*M=I_{K_1}.$
\end{lemma}

\section{Weak TRO equivalence of operator spaces}

In this section we fix  Hilbert spaces $H_1, H_2, K_1, K_2$ and essential reflexive operator spaces
 $$\cl U\subset B(K_1, K_2), \;\;\;\cl V\subset B(H_1 , H_2)$$ which are weak TRO equivalent: i.e., there exist TRO's $\cl M_i\subset 
B(H_i, K_i), i=1,2$ such that 
$$ \cl U=[\cl M_2\cl V\cl M_1^*]^{-w^*} , \cl V=[\cl M_2^*\cl U\cl M_1]^{-w^*} .$$
We assume that 
$$\phi=\mathrm{Map} (\cl{U}), \psi =\mathrm{Map} (\cl{V}), \chi _i=\mathrm{Map} (\cl M_i), \;\;i=1,2, $$$$
\cl B_i=\cl S_{i, \phi }^\prime\subset B(K_i), \cl A_i=\cl S_{i, \psi  }^\prime\subset B(H_i), i=1,2.$$ 
In this section we are going to find *-isomorphisms $\zeta _i: \cl B_i^\prime \rightarrow \cl A_i^\prime, i=1,2$ 
such that if $$\cl N_i=\{T\in B(H_i, K_i): T\zeta _i(P)=PT\;\;\forall \;P\in pr(\cl B_i^\prime )\}, \;\;i=1,2,$$ 
then $$ \cl U=[\cl N_2\cl V\cl N_1^*]^{-w^*} , \cl V=[\cl N_2^*\cl U\cl N_1]^{-w^*} .$$

\begin{lemma}\label{21} $$ \cl A_i=[\cl M_i^*\cl B_i\cl M_i]^{-w^*}, \cl B_i=[\cl M_i\cl A_i\cl M_i^*]^{-w^*}, \;\;i=1,2. $$
\end{lemma} 
\begin{proof} Choose $$Q\in pr(B(H_1))\Rightarrow \psi (Q)\in \cl S_{2,\psi }.$$ 
Since $\cl M_2\cl V\subset \cl U\cl M_1$, we have $\cl M_2\cl VQ\subset \cl U\cl M_1Q.$  The projection onto the space 
generated by $\cl VQ(H_1)$ is $\psi (Q)$ and the projection onto the space 
generated by $\cl U\cl M_1Q(H_1)$ is $ \phi (\chi_1(Q)).$  Thus 
\begin{equation}\label{ex1} \phi (\chi_1(Q))  ^\bot \cl M_2\psi (Q)=0.
\end{equation}
Since $$\cl M_2^*\cl U\cl M_1Q\subset \cl VQ,$$ we have 
\begin{equation}\label{ex2} \psi (Q)^\bot \cl M_2^* \phi (\chi_1(Q)) =0.
\end{equation}
If $B\in \cl B_2, M, N\in \cl M_2,$ then by using (\ref{ex1})  we have 
$$M^*BN\psi (Q)= M^*B\phi (\chi_1(Q)) N\psi (Q) =M^*\phi (\chi_1(Q))BN\psi (Q) .$$
Using (\ref{ex2}), the last operator is equal to $\psi(Q)M^*BN \psi (Q).$ Therefore 
$$\psi (Q)^\bot \cl M_2^*\cl B_2\cl M_2\psi (Q)=0.$$
Since $\cl B_2$ is a selfadjoint algebra, we also have $$\psi (Q) \cl M_2^*\cl B_2\cl M_2 \psi (Q)^\bot =0.$$
Therefore $$\cl M_2^*\cl B_2\cl M_2 \subset \cl S_{2,\psi }^\prime=\cl A_2.$$
Similarly, we can prove that $$\cl M_2\cl A_2\cl M_2 ^*\subset \cl B_2.$$
Proposition 2.1 in \cite{ele} implies that 
$$ \cl A_2=[\cl M_2^*\cl B_2\cl M_2]^{-w^*}, \cl B_2=[\cl M_2\cl A_2\cl M_2^*]^{-w^*}. $$
Similarly, we can prove $$ \cl A_1=[\cl M_1^*\cl B_1\cl M_1]^{-w^*}, \cl B_1=[\cl M_1\cl A_1\cl M_1^*]^{-w^*}. $$

\end{proof}

By Proposition 2.8 in \cite{ele}, the map $$\chi _i^*: pr( \cl B_i^\prime )\rightarrow pr(\cl A_i^\prime ), i=1,2$$ 
extends to a *-isomorphism $\zeta _i: \cl B_i^\prime \rightarrow \cl A_i^\prime , i=1,2$, and if 
$$\cl N_i=\{T\in B(H_i, K_i): T\zeta _i(P)=PT\;\;\forall \;P\in pr(\cl B_i^\prime )\}, \;\;i=1,2,$$ 
then $$ \cl A_i=[\cl N_i^*\cl N_i]^{-w^*},   \;\;\;\cl B_i=[\cl N_i\cl N_i^*]^{-w^*}.   $$
Define the algebras $$ \cl B= \left(\begin{array}{clr}  \cl B_2 & \cl U \\ 0 & \cl B_1\end{array}\right), 
\;\;\;\cl A= \left(\begin{array}{clr}  \cl A_2 & \cl V \\ 0 & \cl A_1\end{array}\right).$$ 
These algebras are reflexive with lattices 
$$ \Lat {\cl B}=\{P_2\oplus P_1: P_i\in pr(\cl B_i^\prime), P_2^\bot \cl UP_1=0\} $$ and 
$$\Lat {\cl A}=\{Q_2\oplus Q_1: Q_i\in pr(\cl A_i^\prime), Q_2^\bot \cl VQ_1=0\} .$$

\begin{lemma}\label{22} The algebras $\cl A, \cl B$ are weak TRO equivalent.
\end{lemma}
\begin{proof} 
We have $$ \left(\begin{array}{clr}  \cl M_2^* &  0 \\ 0 & \cl M_1^*\end{array}\right)
\left(\begin{array}{clr}  \cl B_2 & \cl U \\ 0 & \cl B_1\end{array}\right)
\left(\begin{array}{clr} 
 \cl M_2 &  0 \\ 0 & \cl M_1\end{array}\right)=$$
$$\left(\begin{array}{clr}  \cl M_2^*\cl B_2\cl M_2 & \cl M_2^*\cl U \cl M_1\\ 0 & \cl M_1^*\cl B_1\cl M_1\end{array}\right)
\subset \left(\begin{array}{clr}  \cl A_2 & \cl V \\ 0 & \cl A_1\end{array}\right)=\cl A.$$
Similarly, we can prove that $$(\cl M_2\oplus \cl M_1)\cl A(\cl M_2\oplus \cl M_1)^*\subset \cl B.$$ 
Proposition 2.1 in \cite{ele} implies that  
$$[(\cl M_2\oplus \cl M_1)\cl A(\cl M_2\oplus \cl M_1)^*]^{-w^*}= \cl B$$ and  
 $$[(\cl M_2\oplus \cl M_1)^*\cl B(\cl M_2\oplus \cl M_1)]^{-w^*}=\cl A.$$ 

\end{proof}

By Theorem 3.3 in \cite{ele}, if $\chi =\mathrm{Map}(\cl M_2\oplus \cl M_1)$, then
$\chi ^*: pr(\cl B_2^\prime\oplus \cl B_1^\prime)\rightarrow  pr(\cl A_2^\prime\oplus \cl A_1^\prime)$ extends to a *-isomorphism 
$$\zeta : \cl B_2^\prime\oplus \cl B_1^\prime\rightarrow  \cl A_2^\prime\oplus \cl A_1^\prime $$ such that 
$\zeta (\Lat{\cl B})=\Lat{\cl A}.$ Observe that $\zeta =\zeta_2 \oplus \zeta_1. $ 
Also by Theorem 3.3 in \cite{ele}, the TRO 
$$\cl N=\{T: T\zeta (P)=PT\;\;\forall \;P\in pr(\cl B_2^\prime\oplus \cl B_1^\prime )\}=\cl N_2\oplus \cl N_1$$ 
 implements a TRO equivalence of the algebras $\cl A$ and $\cl B:$
$$ [\cl N\cl A\cl N^*]^{-w^*}=\cl B ,\;\;\; [\cl N^*\cl B\cl N]^{-w^*}=\cl A. $$

\begin{theorem}\label{23} $$ \cl U=[\cl N_2\cl V\cl N_1^*]^{-w^*} , \cl V=[\cl N_2^*\cl U\cl N_1]^{-w^*} .$$
\end{theorem}
\begin{proof} We have 
$$
\left(\begin{array}{clr}  \cl N_2 & 0 \\ 0 & \cl N_1\end{array}\right)
\left(\begin{array}{clr}  0  & \cl V \\ 0 & 0 \end{array}\right)\left(\begin{array}{clr}  \cl N_2^* & 0 \\ 0 & \cl N_1^*
\end{array}\right)\subset \cl B.$$ 
Thus $$ \cl N_2\cl V\cl N_1^* \subset \cl U.$$
Similarly, $$ \cl N_2^*\cl U\cl N_1\subset \cl V  \Rightarrow \cl N_2\cl N_2^*\cl U\cl N_1\cl N_1^*\subset \cl N_2\cl V \cl N_1^* .$$
Since the algebras $[\cl N_i\cl N_i^*]^{-w^*}$ contain the identity operators, the last relation implies that 
$$ \cl U\subset [\cl N_2\cl V\cl N_1^*]^{-w^*}\Rightarrow  \cl U= [\cl N_2\cl V\cl N_1^*]^{-w^*}.$$
Similarly, we can prove $$\cl V= [\cl N_2^*\cl U\cl N_1]^{-w^*}.$$

\end{proof}

\section {Hyperreflexivity and weak TRO equivalence}

In this section we fix  Hilbert spaces $H_1, H_2 ,K_1, K_2$ and essential weak* closed spaces 
$$\cl U\subset B(K_1, K_2), \;\;\;\cl V\subset B(H_1 , H_2)$$ that are weak TRO equivalent. 
We fix an infinite  cardinal  $I$ greater than or equal to the maximum of $I_i, i=1,2$ where $I_i$ is the cardinal 
of an orthonormal basis of $H_i.$ We are going to prove that $k( M_I^w(\cl V) )\leq k(M_I^w(\cl U) ).$  
If $k(M_I^w(\cl U) )=\infty   $, the inequality is obvious. So we assume throughout this section that  $k(M_I^w(\cl U) )<\infty   .$ 
From the results of Section 2, there exist von Neumann algebras $\cl B_i\subset B(K_i), \cl A_i\subset B(H_i), i=1,2$
 and *-isomorphisms $\zeta _i: \cl B_i^\prime \rightarrow \cl A_i^\prime, i=1,2$ 
such that if $$\cl N_i=\{T\in B(H_i, K_i): T\zeta _i(P)=PT\;\;\forall \;P\in pr(\cl B_i^\prime )\}, \;\;i=1,2,$$ 
then $$ \cl U=[\cl N_2\cl V\cl N_1^*]^{-w^*} , \cl V=[\cl N_2^*\cl U\cl N_1]^{-w^*} , \cl A_i=[\cl N_i^*\cl N_i]
^{-w^*}, \cl B_i=[\cl N_i\cl N_i^*]^{-w^*},  i=1,2.$$
We also recall the algebras $\cl A, \cl B$ defined in Section 2. Since $\cl U$ is a hyperreflexive space, $\cl B$ is a reflexive algebra and thus by 2.7.i 
in \cite{ele}, $\cl A$ is also a reflexive algebra. Therefore $\cl V$ is a reflexive space.

\begin{lemma}\label{31} $$\cl V= \{T\in B(H_1, H_2): P_i\in pr(\cl B_i^\prime) ,i=1,2, P_2\cl UP_1=0\Rightarrow \zeta_2(P_2) 
T\zeta _1(P_1)=0\}.$$
\end{lemma}
\begin{proof} We denote by $\Omega $ the space 
$$\{T\in B(H_1, H_2): P_i\in pr(\cl B_i^\prime) ,i=1,2, P_2\cl UP_1=0\Rightarrow \zeta_2(P_2) 
T\zeta _1(P_1)=0\}.$$
Fix $P_i\in pr(\cl B_i^\prime) ,i=1,2$ such that $ P_2\cl UP_1=0. $  We recall $\zeta , \cl A, \cl B$ from 
Section 2. We have that $P_2^\bot \oplus P_1\in \Lat{\cl B}.$ Since 
$\zeta ( \Lat{\cl B} ) = \Lat{\cl A} $ we take $\zeta_2(P_2)^\bot \oplus  \zeta_1(P_1)\in \Lat{\cl A} . $ Therefore 
$ \zeta_2(P_2) \cl V \zeta_1(P_1)=0.$ It follows that $\cl V\subset \Omega .$

Conversely, if $T\in \Omega $ and $P_2^\bot \cl UP_1=0$ for $P_i\in pr(\cl B_i^\prime), i=1,2$, then 
$$
\left(\begin{array}{clr}  \zeta_2(P_2)^\bot    & 0 \\ 0 & \zeta_1(P_1) ^\bot \end{array}\right)
\left(\begin{array}{clr}  0   & T \\ 0 & 0 \end{array}\right)\left(\begin{array}{clr}  \zeta_2(P_2) 
  & 0 \\ 0 & \zeta_1(P_1)  \end{array}\right)=0, \;\;\forall \;\;P_2 \oplus P_1\in \Lat{\cl B} .$$ 
 Therefore $$\zeta (Q)^\bot (0\oplus T)\zeta (Q)=0\;\;\forall \;\;Q\;\in \;\Lat{\cl B}.$$ 
But $\zeta (\Lat{\cl B})= \Lat{\cl A}. $
Thus $$Q^\bot (0\oplus T)Q=0\;\;\forall \;\;Q\;\in \;\Lat{\cl A}.$$ 

 Therefore $$\left(\begin{array}{clr}  0  & T \\ 0 & 0 \end{array}\right)
\in \cl A\Rightarrow T\in \cl V.$$ We have thus proved $\Omega \subset \cl V\Rightarrow \Omega =\cl V.$
\end{proof}

We define the space $$\cl W=[\cl V\cl N_1^*]^{-w^*}\subset B(K_1, H_2).$$

\begin{lemma}\label{32} $$\cl W= \{T\in B(K_1, H_2): P_i\in pr(\cl B_i^\prime) ,i=1,2, P_2\cl UP_1=0\Rightarrow \zeta_2(P_2) 
TP_1=0\}.$$\end{lemma}
\begin{proof} Define $$\Omega =\{T\in B(K_1, H_2): P_i\in pr(\cl B_i^\prime) ,i=1,2, P_2\cl UP_1=0\Rightarrow \zeta_2(P_2) 
TP_1=0\}.$$
Fix $P_i\in pr(\cl B_i^\prime) ,i=1,2$ such that $ P_2\cl UP_1=0$ and fix $V\in \cl V, S\in \cl N_1.$ 
We have $$ \zeta _2(P_2)VS^*P_1 = \zeta _2(P_2)V\zeta _1(P_1) S^*. $$ By Lemma \ref{31},
$\zeta _2(P_2)V\zeta _1(P_1) =0.$ Thus  $\zeta _2(P_2)VS^*P_1=0.$ We have thus proved $\cl W\subset \Omega .$ 

For the converse, fix $A\in \Omega $ and $S\in \cl N_1. $ If $P_i\in pr(\cl B_i^\prime) ,i=1,2$ such that $ P_2\cl UP_1=0$ 
we have  $$ \zeta_2(P_2) AS \zeta_1(P_1)=  \zeta_2(P_2) AP_1S=0S=0. $$
Thus $$ \Omega \cl N_1\subset \cl V \Rightarrow \Omega \cl N_1\cl N_1^*\subset \cl W .$$ 
But $[\cl N_1\cl N_1^*]^{-w^*}$ contains the identity operator. Therefore $\Omega \subset \cl W.$ The proof is complete.

\end{proof}

\begin{lemma}\label{33} $$k(\cl W)\leq k(M_I^w(\cl U)).$$
\end{lemma}
\begin{proof}  Suppose that $I_2$ is the cardinal of an orthonormal basis of $H_2.$ We have $I_2\leq I.$  
By Lemma \ref{111}, there exists a column $N=(N_i)_{i\in I_2}$ such that and $N_i\in \cl N_2$ for all $i,$
 and $ N^*N=I_{H_2}.$  Adding zeros, if necessary, we may assume that $N=(N_i)_{i\in I}.$ We claim that 
$$\cl W=N^*  C_I ^w(\cl U) .$$ Indeed 
$$ \cl N_2^*\cl U= [\cl N_2^*\cl N_2\cl V\cl N_1^*]^{-w^*} \subset [\cl V\cl N_1^*]^{-w^*} =\cl W.$$
Thus $$N^*C_I^w(\cl U)\subset \cl W.$$
Since $ N\cl W\subset C_I ^w(\cl U)$, we have $$ N^*N \cl W\subset N^*C_I ^w(\cl U)\Rightarrow 
\cl W\subset N^* C_I ^w(\cl U) .$$ So the claim holds. 
In the sequel we use the fact $$k( C_I ^w(\cl U) )\leq  k( M_I ^w(\cl U) ) .$$

Fix $A\in B(K_1, H_2). $ We have 
\begin{align*} d(A, \cl W)=&  \inf_{W\in \cl W}\|A-W\| = \inf_{U\in C_I ^w(\cl U) }\|A-N^*U\| =\\&
\inf_{U\in  C_I ^w(\cl U)  }\|N^*NA-N^*U\|  \leq \inf_{U\in C_I ^w(\cl U) }\|NA-U\|  =\\&
d(NA,  C_I ^w(\cl U)  )\leq  k( M_I ^w(\cl U) ) r_{ C_I ^w(\cl U)  }(NA).
\end{align*}
Since $\cl U$ is a $\cl B_2-\cl B_1$ bimodule, $ C_I ^w(\cl U)  $ is a $M_I^w(\cl B_2) -\cl B_1$ bimodule. 
Therefore, for any $\epsilon >0$, there exist  $P_i\in pr(\cl B_i^\prime)$ such that  $P_2\cl UP_1=0$ and 
 \begin{align*} r_{ C_I ^w(\cl U)  }(NA)-\epsilon <& \|P_2^I NAP_1\|=         \|N\zeta _2(P_2)AP_1\|     \leq \\& 
 \| \zeta _2(P_2)AP_1 \|  .   \end{align*} 
By Lemma \ref{32}, $\zeta _2(P_2)\cl WP_1 =0,$  thus $$\| \zeta _2(P_2)AP_1 \|  \leq r_\cl W(A).$$ 
Since $\epsilon $ is arbitrary, $$ r_{ C_I ^w(\cl U)  }(NA)\leq r_\cl W(A).$$  We have thus proved that 
$$d(A, \cl W)\leq  k(M_I^w(\cl U)) r_\cl W(A).$$   The proof is complete.

\end{proof}

\begin{lemma}\label{34} $$k(M_I^w(\cl W)) \leq k(M_I^w(\cl U)) $$
\end{lemma}
\begin{proof} We can see that the spaces $ M_I ^w(\cl U) , M_I ^w(\cl V) $ are weak TRO equivalent: 
$$M_I ^w(\cl U)=[ M_I ^w(\cl N_2) M_I ^w(\cl V)M_I^w(\cl N_1)^*]^{-w^*}, $$$$
 M_I ^w(\cl V)=[ M_I ^w(\cl N_2)^* M_I ^w(\cl U)M_I ^w(\cl N_1)]^{-w^*}  .$$ 
Following the above arguments from the beginning to Lemma \ref{33}, 
the space $$M_I ^w(\cl W) =[M_I ^w(\cl V) M_I ^w(\cl N_1)^*]^{-w^*} $$ 
has  hyperreflexivity constant less than or equal to  the hyperreflexivity constant of $M_I^w(M_I^w(\cl U)) $ 
which, since $I$ is infinite, is equal to the hyperreflexivity constant of $M_I^w(\cl U).$
\end{proof}

\begin{lemma}\label{35} $$k(\cl V)\leq k(M_I^w(\cl U)).$$
\end{lemma}
\begin{proof} Let $I_1$ be the cardinal of an orthonormal basis of $H_1.$ 
We can find an infinite column $M=(M_i)_{i\in I_1}, M_i\in \cl N_1$ such that 
$M^*M=I_{H_1},$ (Lemma \ref{111}). Adding zeros if necessary, we may assume that $M=(M_i)_{i\in I}.$ We have 
$$\cl V \cl N_1^*\cl N_1 \subset \cl V\Rightarrow \cl N_1^*\cl N_1 \cl V^*\subset \cl V^*\Rightarrow \cl N_1^*\cl W^*\subset \cl V^*.$$ 
Therefore $$M^* C_I ^w(\cl W^*) \subset \cl V^*.$$
On the other hand, $\cl V^*=M^*M\cl V^*.$ Since $M\cl V^*\subset C_I ^w(\cl W^*)$, we have 
$$\cl V^*\subset M^*C_I ^w(\cl W^*) \Rightarrow \cl V^*=M^*C_I ^w(\cl W^*) .$$
Choose $T\in B(H_2, H_1)$.  Using Lemma \ref{34}, we have
\begin{align*} d(T, \cl V^*)=  & \inf_{V\in \cl V}\|T-V^*\|  = \inf_{S\in C_I ^w(\cl W^*) 
}\|T-M^*S\| =  \inf_{S\in C_I ^w(\cl W^*) 
}\|M^*MT-M^*S\| \leq\\ & \inf_{S\in C_I ^w(\cl W^*) 
}\|MT-S\|\leq k(M_I^w(\cl U)) r_{ C_I ^w(\cl W^*)  }(MT).\end{align*} 
Fix $\epsilon >0.$ Since $\cl W$ is an $\cl A_2-\cl B_1$ bimodule, there exist $P\in pr(\cl A_2^\prime), Q\in 
pr(\cl B_1^\prime)$ such that $QW^*P=0$ and $$  r_{C_I ^w(\cl W^*)  }(MT)-\epsilon   <\|Q^I MTP\|=
 \|M\zeta _1(Q)TP\|   \leq  \|\zeta _1(Q)TP\| .  $$
We have 
$$P\cl WQ=0\Rightarrow P\cl V\cl N_1^*Q=0\Rightarrow P\cl V\zeta _1(Q)=0\Rightarrow \zeta _1(Q)\cl V^*P=0.$$
Therefore $$ r_{C_I ^w(\cl W^*)  }(MT) -\epsilon   < r_{\cl V^*}(T) .$$ Since 
$\epsilon $ is arbitrary, we have 
$$r_{C_I ^w(\cl W^*)  }(MT) \leq r_{\cl V^*}(T) \Rightarrow d(T, \cl V^*)\leq k(M_I^w(\cl U)) r_{\cl V^*}(T) .$$
Therefore $\cl V^*$, and hence $\cl V$ has  hyperreflexivity constant less than $k(M_I^w(\cl U)) .$

\end{proof}

\begin{theorem}\label{36} Let $\cl U, \cl V, H_1, H_2, K_1, K_2, I$ be as in the beginning of this section. 
Then $$k(M_I^w(\cl V))\leq k(M_I^w(\cl U)).$$ In the special case that $H_1, H_2, K_1, K_2$ are separable, we have 
$$k(M_\infty ^w(\cl V))= k(M_\infty ^w(\cl U)).$$

\end{theorem}
\begin{proof}  The spaces $ M_I ^w(\cl U) , M_I ^w(\cl V) $ are weak TRO equivalent: 
$$M_I^w(\cl U)=[ M_I ^w(\cl N_2) M_I ^w(\cl V) M_I ^w(\cl N_1)^*]^{-w^*}, $$$$
 M_I ^w(\cl V)=[ M_I ^w(\cl N_2)^* M_I ^w(\cl U)M_I ^w(\cl N_1)]^{-w^*}  .$$ 
Following the arguments from the beginning to Lemma \ref{35}, $M_I ^w(\cl V)$ has hyperreflexivity constant less than or equal to  
$k(M_I^w(M_I^w(\cl  U))=k(M_I^w(\cl U)).$ 

If $\cl U, \cl V$ are separably acting spaces, then by the first part of the proof, $$k(M_\infty ^w(\cl V))\leq  k(M_\infty ^w(\cl U)).$$ 
By symmetry $$k(M_\infty ^w(\cl U))\leq  k(M_\infty ^w(\cl V)).$$ 
\end{proof}

\section{Isomorphisms and complete hyperreflexivity.}

In this section, for each reflexive space $\cl X$ we write  $k_c(\cl X)$ for its complete hyperreflexivity 
constant, $k(M_\infty ^w(\cl X).$

\begin{theorem}\label{42} Let $\cl B, \cl A$ be stably isomorphic CSL algebras acting on the separable Hilbert spaces $K, H$ 
respectively. If $\cl B$ is completely hyperreflexive, then $\cl A$ is also completely hyperreflexive and $k_c(B)=k_c(A).$
\end{theorem}
\begin{proof} By Theorem 3.2 in \cite{ele2} and the main result of \cite{ep}, the algebras $\cl B$ and $\cl A$ are weak TRO equivalent. The conclusion comes from Theorem \ref{36}.
\end{proof}

\begin{corollary}\label{43} Let  $\cl B, \cl A$ be  CSL algebras acting on the separable Hilbert spaces $K, H$ 
respectively. We assume that  $\cl B$ is completely hyperreflexive. If either 

(i) $\cl A$ is not completely hyperreflexive, or 

(ii) $\cl A$ is completely hyperreflexive, but $k_c(\cl A)\neq  k_c(\cl B)$,

\noindent then $\cl B$ and $\cl A$ cannot be stably isomorphic.
\end{corollary}

\begin{remark}\label{44} In view of Theorem \ref{42}, we remark that  two stably isomorphic completely hyperreflexive spaces need not have 
the same complete hyperreflexivity constant. For example, take $H=l^2(\mathbb N), I_H$ the corresponding identity operator, 
$\cl X=\mathbb C I_H$ and $\cl Y=B(H).$ Since the space $\cl X\bar \otimes B(H)$ is  isomorphic as a dual operator space 
with  $\cl Y\bar \otimes B(H), \cl X$ and $\cl Y$ are stably isomorphic. But $k_c(\cl Y)=1$  and $k_c(\cl X)>1.$ 
(See Lemma 6.11 in \cite{dale}). 

\end{remark}

\begin{theorem}\label{45}  Let $\cl L_i\subset B(H_i), i=1,2$ be separably acting reflexive lattices. If there exists a *-isomorphism 
$\theta : \cl L_1^{\prime \prime} \rightarrow \cl L_2^{\prime \prime} $ such that $\theta (\cl L_1)=\cl L_2$, then the algebras 
$ \Alg{\cl L_1} , \Alg{\cl L_2} $ are weak TRO equivalent,  (Theorem 3.3 in \cite{ele}). Therefore,
by Theorem \ref{36}, $k_c(\Alg{\cl L_1})=k_c(\Alg{\cl L_2} ).$  
\end{theorem} 

\begin{corollary}\label{45a} Let $\cl L_i\subset B(H_i), i=1,2$ be separably acting totally atomic CSL's. If these lattices are isomorphic 
as CSL's, then the algebras 
$ \Alg{\cl L_1} , \Alg{\cl L_2} $ are weak TRO equivalent,  (Theorem 5.3 in \cite{ele}). Therefore, by Theorem \ref{36},
  $k_c(\Alg{\cl L_1})=k_c(\Alg{\cl L_2} ).$   
\end{corollary}

\begin{theorem} \label{37} Let $\cl A, \cl B$ be von Neumann algebras acting on the Hilbert spaces $K$ and $H$ 
respectively. If $\pi : \cl A^\prime\rightarrow \cl B^\prime$ is a $*$-isomorphism and $I$ is an infinite cardinal greater than or equal to the cardinal 
of an orthonormal basis of $H$, then $$ k(M_I^w(\cl B)) \leq k(M_I^w(\cl A)). $$ 
 In the special case where $\cl A, \cl B$ are separably acting, we have  $$k(M_\infty ^w(\cl B))= k(M_\infty ^w(\cl A)).$$ 
\end{theorem}
\begin{proof} We define the following TRO:
$$\cl M=\{M\in B(K, H): MA=\pi (A)M \;\;\forall \;A\;\in \;\cl A\}.$$
By Theorem 3.2 in \cite{ele}, we have $$[\cl M^*\cl M] ^{-w^*}=\cl B ,\;\;\;  [\cl M\cl M^*] ^{-w^*}=\cl A.$$ 
Thus $\cl A, \cl B$ are weak TRO equivalent in the sense of this paper. The conclusion comes from Theorem \ref{36}.
\end{proof}

\begin{theorem}\label{48} Let $\cl A$ be a separably acting von Neumann algebra for which the commutant is stable, 
i.e., $\cl A^\prime$  and $M_\infty ^w(\cl A^\prime)$ are isomorphic. Then $\cl A$ is completely hyperreflexive and $k_c(\cl A)\leq 9.$
\end{theorem}
\begin{proof} The algebra $M_\infty ^w(\cl A^\infty )$ is unitarily equivalent to $M_\infty ^w(\cl A )^\infty .$ 
The last algebra is hyperreflexive with constant less than $9,$ \cite{kl}. Thus $k(M_\infty ^w(\cl A^\infty ))\leq 9.$  
Since by hypothesis the commutants $(\cl A^\infty )^\prime$ and $\cl A^\prime$ are isomorphic, Theorem \ref{37}  
implies that $$k(M_\infty ^w(\cl A))=k(M_\infty ^w(\cl A^\infty ))\leq 9.$$
\end{proof}


\begin{thebibliography}{99}




\bibitem{bm}
D.P. Blecher and C. Le Merdy, {\em Operator Algebras and Their Modules}, London Mathematical Society, 2004.

\bibitem{dav}
K.R. Davidson,
Triangular forms for operator algebras on Hilbert space.
\newblock In: {\em Nest Algebras}, volume 191 of {\em Pitman
Research Notes in
  Mathematics Series}.
\newblock Longman, Harlow, UK, 1988.

\bibitem{dale}
 K.R. Davidson and R.H. Levene,
 1-hyperreflexivity and complete hyperreflexivity, {\em J. Funct. Anal.}, 
235 (2005), 666--701.


\bibitem{ele}
G.K. Eleftherakis, TRO equivalent algebras, 
{\em Houston J. of Mathematics}, 38 (2012) no. 1, 153--175.

\bibitem{ele2}
G.K. Eleftherakis,  Morita type equivalences and reflexive
algebras, {\em J. Operator Theory}, 64 (2010) no. 1, 3--17.



\bibitem{ep}
G.K.  Eleftherakis and  V.I. Paulsen, Stably isomorphic dual operator algebras,  
{\em Math. Annalen }, 341 (2008), no 1, 99--112.

\bibitem{ept}
G.K. Eleftherakis, V.I. Paulsen, and I.G. Todorov, Stable isomorphism of dual operator spaces, {\em J. Funct. Anal.} 
258  (2010), 260--278.

\bibitem{erd}
J.A. Erdos, Reflexivity for subspace maps and linear spaces of operators, {\em Proc. London Math. Soc.} 52 (3), (1986),  582--600.


\bibitem{kt} A. Katavolos and I.G. Todorov, Normalisers of
operator algebras and reflexivity, {\em Proc. London Math. Soc.}
(3) 86, (2003),  463--484.


\bibitem{kl} J. Kraus and  D. Larson, Reflexivity and distance formulae, {\em Proc. London Math. Soc.} 
53, (1986), 340--356.

\end{thebibliography}
\end{document}